\documentclass{amsart}
\usepackage{amssymb}

\newtheorem{thm}[equation]{Theorem}
\newtheorem{lem}[equation]{Lemma}
\newtheorem{lem-dfn}[equation]{Lemma-Definition}
\newtheorem{prop}[equation]{Proposition}
\newtheorem{cor}[equation]{Corollary}

\newtheorem{ass}[equation]{Assumption}

\theoremstyle{definition}
\newtheorem{defn}[equation]{Definition}

\newtheorem{ex}[equation]{Example}

\newtheorem{quest}[equation]{Question}

\newtheorem{prob}[equation]{Problem}

\newtheorem*{acknowledgement}{Acknowledgement}
\theoremstyle{remark}
\newtheorem{clm}{Claim}
\newtheorem{rem}[equation]{Remark}
\numberwithin{equation}{section}
\newcommand{\thmref}[1]{Theorem~\ref{#1}}
\newcommand{\lemref}[1]{Lemma~\ref{#1}}

\newcommand{\proref}[1]{Proposition~\ref{#1}}

\newcommand{\clmref}[1]{Claim~\ref{#1}}
\newcommand{\defref}[1]{Definition~\ref{#1}}

\newcommand{\figref}[1]{Figure~\ref{#1}}
\newcommand{\sref}[1]{Section~\ref{#1}}

\DeclareMathOperator{\Ext}{Ext}
\DeclareMathOperator{\Hom}{Hom}

\DeclareMathOperator{\rank}{rank}

\DeclareMathOperator{\emb}{embdim}

\DeclareMathOperator{\ord}{ord}

\DeclareMathOperator{\coeff}{coeff}
\DeclareMathOperator{\pic}{Pic}

\DeclareMathOperator{\di}{div}

\DeclareMathOperator{\id}{Id}

\DeclareMathOperator{\mult}{mult}

\newcommand{\m}{\mathfrak m}

\newcommand{\Z}{\mathbb Z}
\newcommand{\Q}{\mathbb Q}

\newcommand{\C}{\mathbb C}

\newcommand{\cI}{\mathcal I}

\newcommand{\cO}{\mathcal O}
\newcommand{\cP}{\mathcal P}

\newcommand{\tX}{\widetilde X}

\renewcommand{\:}{\colon}

\newcommand{\defset}[2]{{\left\{#1\,\left| \,#2 \right. \right\}}}
\newcommand{\X}{(X,o)}

\newcommand{\zmi}{Z_{\rm{min}}}
\newcommand{\zma}{Z_{\rm{max}}}
\newcommand{\zco}{Z_{\rm{coh}}}

\title[Analytic structures supported  on a fixed topological type]{Analytic singularities supported
by a specific  integral homology sphere link}

\author{Andr\'as N\'emethi}
\thanks{The first author was partially supported by NKFIH Grant  112735 and
ERC Adv. Grant LDTBud of A. Stipsicz at R\'enyi Institute of Math., Budapest}
\address{Alfr\'ed R\'enyi Institute of Mathematics,
Hungarian Academy of Sciences,
Re\'altanoda utca 13-15, H-1053, Budapest, Hungary \newline
 \hspace*{4mm} ELTE - University of Budapest, Hungary \newline \hspace*{4mm}
BCAM - Basque Center for Applied Mathematics, Bilbao, Spain}
\email{nemethi.andras@renyi.mta.hu }
\author{Tomohiro Okuma}
\thanks{The second author was partially supported by JSPS  KAKENHI
Grant Number 26400064.}
\address{Department of Mathematical Sciences,
Faculty of Science, Yamagata University, Yamagata, 990-8560, Japan.}
\email{okuma@sci.kj.yamagata-u.ac.jp}



\begin{document}

\begin{abstract}
The main question we target is the following:
If one fixes a topological type (of a complex normal surface singularity)
then what are the possible analytic types supported by it, and/or,
what are the  possible values of the geometric genus?
We answer the question for a specific (in some sense pathological)
topological type, which supports
 rather different analytic structures.
These structures are listed together with some of their key analytic invariants.
\end{abstract}

\maketitle

\begin{center}
{\it Dedicated to Henry Laufer  on the occasion of his 70th birthday}
\end{center}

\setcounter{tocdepth}{1}
%\tableofcontents

\section{Introduction}

The topological type of  a normal complex surface singularity $(X,o)$ is determined by its
link (an oriented smooth connected 3--manifold), or, by the dual graph of any good resolution
(a  connected graph with a negative definite intersection form
 \cite{GRa,Mu}, which serves also as
 plumbing graphs of the link \cite{neumann}).

The main question we target is the following:
\begin{quest}
If one fixes a topological type (say, a minimal good resolution graph)
and varies the possible analytic types supported on this fixed
topological type, then what are the possible values of the geometric genus $p_g$?
\end{quest}
Slightly more concrete version is formulated as follows:

\begin{prob}
Associate combinatorially an integer ${\rm MAX}(\Gamma)$ to any (resolution) graph $\Gamma$, such that
for any analytic type supported by $\Gamma$ one has $p_g\leq {\rm MAX}(\Gamma)$, and furthermore,
for certain analytic structure one has equality.

Moreover, define by symmetric properties    ${\rm MIN}(\Gamma)$ as well.
\end{prob}
 A possible topological lower bound for $p_g$ can be constructed as follows. Fix a  resolution $\tX\to X$
 and for
 any divisor $l$  supported by the exceptional divisor set $\chi(l):= -(l, l-Z_K)/2$, where $Z_K$ is the
 anti-canonical cycle (see below) and $( \, , \, )$ denotes the intersection form. Set also ${\rm min}\chi$ as $\min_{l}\chi(l)$. Then ${\rm min}\chi$ is
 a topological invariant computable from $\Gamma$; Wagreich considered the expression
 $p_a(X,o)=1-\min\chi$, and called it the `artihmetical genus' \cite{Wael}.
 Moreover, for any analytic structure,
 whenever $p_g>0$,  one also has (see e.g. \cite[p. 425]{Wael})
 \begin{equation}\label{eq:min}
 1-{\rm min}\chi \leq p_g.
 \end{equation}
 Indeed, one verifies that ${\rm min}\chi$ can be realized by an effective cycle $l_0>0$ (see e.g.
 \cite{Book}). Then from the cohomological long exact sequence associated with
 $0\to \cO_{\tX}(-l_0)\to\cO_{\tX}\to \cO_{l_0}\to 0$ one has
 $$p_g+\chi(l_0)=\dim H^0(\cO_{\tX})/H^0(\cO_{\tX}(-l_0))+h^1(\cO_{\tX}(-l_0))\geq 1$$
 (since $H^0(\cO_{\tX})/H^0(\cO_{\tX}(-l_0))$ contains the class of constants).
 (\ref{eq:min}) sometimes is sharp: e.g.
 for elliptic singularities (when ${\rm min}\chi=0$) Laufer proved that for the generic
 analytic structure one has indeed
  $p_g=1-{\rm min}\chi=1$ \cite{la.me}.

For different generalizations of (\ref{eq:min})
(inequalities, which involve  besides $\min\chi$ and $p_g$ some other analytic invariants as well) see e.g. \cite[(2.6)]{Tomari86} or \cite[Prop. 8]{KNF}.

 However, the authors do not know if the above  bound (\ref{eq:min}) is always optimal:
 \begin{quest} Does there exist for any $\Gamma$ an analytic structure
  with $p_g=1-{\rm min}\chi$?\end{quest}
 A possible upper bound for $p_g$ is constructed as follows \cite{nem.lattice}.

  Let $\{E_i\}_{i\in\cI}$  denote the set of
 irreducible exceptional curves, and for simplicity {\it
 we will assume that each $E_i$ is rational}.
 For any effective cycle $Z>0$  let $\cP(Z)$ be the set of
  monotone computation sequences $\gamma=\{l_k\}_{k=0}^t$
  of cycles supported on the exceptional curve with the following properties:
 $l_0=0$, $l_t=Z$,  and $l_{k+1}=l_k+ E_{i(k)}$ for some $i(k)\in \cI$.
 Associated with such $\gamma$ we define
 $$S(\gamma):=\sum_{k=0}^{t-1} \max\{0, (E_{i(k)},l_k)-1\}.$$
 Set also ${\rm Path}(Z):=\min _{\gamma\in\cP(Z)} S(\gamma)$.
 Then for any analytic structure supported on $\Gamma$ one has
 \begin{equation}\label{eq:maxZ}
h^1(\cO_Z)\leq {\rm Path}(Z).
 \end{equation}
Indeed, from the
exact sequence $0\to \cO_{E_{i(k)}}(-l_k)\to \cO_{l_{k+1}}\to \cO_{l_k}\to 0$  we get
\begin{equation*}
h^1(\cO_{l_{k+1}})-h^1(\cO_{l_k})\leq h^1(\cO_{E_{i(k)}}(-l_k)) = \max\{0, (E_{i(k)},l_k)-1 \}
\ \ \ \ \ (0\leq k<t),
\end{equation*}
%Since $h^1(\cO_{l_t})=p_g$,
hence the inequality follows by summation.
Since $p_g=h^1(\cO_{\lfloor Z_K \rfloor})=h^1(\cO_Z)$ for any $Z\geq \lfloor
Z_K\rfloor $ when $Z_K\ge 0$, it is natural to define
${\rm Path}(\Gamma):= \min_{Z\geq \lfloor Z_K\rfloor } {\rm Path}(Z)$.
It satisfies
  \begin{equation}\label{eq:max}
p_g \leq {\rm Path}(\Gamma
).
 \end{equation}

The computation of ${\rm Path}(\Gamma)$ is rather hard. In \cite{nem.lattice} (see also \cite{N-Sig})
is related with the Euler characteristic of the `path lattice cohomology' of $\Gamma$.
In the next statement we collect some families of singularities when (\ref{eq:max}) is sharp.
\begin{thm} In the next statement we consider singularities with 
 rational homology sphere link. 
In the following cases $p_g={\rm Path}(\Gamma)$ (hence these analytic families
realize the maximal $p_g$ on their topological type):

-- \ weighted homogeneous normal surface singularities \cite{Book}
(in fact, for star shaped graphs with all $E_i$ rational, ${\rm Path}(\Gamma)$
equals the topological expression of Pinkham valid for $p_g$ \cite{pinkham}),

-- \ superisolated hypersurface singularities \cite{N-Sig},

-- \ isolated hypersurface  Newton--nondegenerate singularities  \cite{N-Sig},

-- \ rational singularities \cite{Book},

-- \ Gorenstein elliptic singularities \cite{Book}.
\end{thm}

One can expect that the realization $p_g={\rm Path}(\Gamma)$ is even more general.

%(The interpretation of the first identity is the following. Take a star
%shaped graph with all Euler
%decorations zero. Then for any analytic type with this resolution graph one has $p_g\leq {\rm Path }(\Gamma)$, and equality holds for the weighted homogeneous singularities. This also shows that ${\rm Path}(\Gamma)$
%equals the topological expression of Pinkham valid for $p_g$ for weighted homogeneous  germs \cite{pinkham}.
%For other analytic families on the above list we have similar interpretation: they are those families
%for which the geometric genus realizes the maximal value on the corresponding graphs.) \marginpar{This sentence in ( \ \ ) can be omitted or should be simplified. The point is the Path coincides with the Pinkham's $p_g$-formula for $C^*$-action.}

However,
 the main aim of the present article is to show that the upper bound (\ref{eq:max}) in general is not sharp:
 for certain graph $\Gamma$ the bound ${\rm Path}(\Gamma)$ cannot be realized.
 Surprisingly, the very same example shows  some additional statements as well:
 (the third part is motivated by the `conviction'
  that usually  `large'  $p_g$  is realized simultaneously with `small' maximal ideal cycle):

 \begin{thm}\label{t:pg<path} There exists a numerically Gorenstein topological type for which

 -- \ $p_g<{\rm Path}(\Gamma)$ for any analytic type supported on $\Gamma$;

 -- \ even if an analytic type realizes the maximal $p_g$ (among all analytic types supported on the topological type under discussion)
 it is not necessarily Gorenstein;

 -- \ even if an analytic type realizes the maximal $p_g$,  the maximal ideal cycle is not necessarily the
 Artin cycle.
 \end{thm}
 %\marginpar{The second and third claims in Thm 1.7 seems to be modified:
%How about the following:
%
%(1) an analytic type which realizes the maximal $p_g$  is not necessarily Gorenstein,
%then it is not necessarily Gorenstein,...
%
%(2) even if an analytic type realizes the maximal $p_g$, it is not necessarily Gorenstein,...}

Our fixed topological type, which has the above properties, is given by the minimal good graph
from Figure 1.

\begin{figure}
\begin{picture}(300,45)(30,0)
\put(125,25){\circle*{4}}
\put(150,25){\circle*{4}}
\put(175,25){\circle*{4}}
\put(200,25){\circle*{4}}
\put(225,25){\circle*{4}}
\put(150,5){\circle*{4}}
\put(200,5){\circle*{4}}
\put(125,25){\line(1,0){100}}
\put(150,25){\line(0,-1){20}}
\put(200,25){\line(0,-1){20}}
\put(125,35){\makebox(0,0){$-3$}}
\put(150,35){\makebox(0,0){$-1$}}
\put(175,35){\makebox(0,0){$-13$}}
\put(200,35){\makebox(0,0){$-1$}}
\put(225,35){\makebox(0,0){$-3$}}
\put(160,5){\makebox(0,0){$-2$}}
\put(210,5){\makebox(0,0){$-2$}}
\end{picture}
\caption{The graph $\Gamma$}
\label{fig:gamma}
\end{figure}
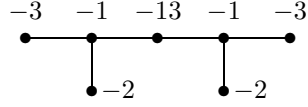

In the next statements we assume
 that $\X$ has the resolution graph $\Gamma$ from Figure 1 and $\tX$ is its minimal good resolution.
 Let $\zmi$ be the Artin cycle, while $\zma$ the maximal ideal cycle
 introduced by  S. S.-T. Yau \cite{Yau1} (see \defref{d:cycles}).
 For this graph one has ${\rm min}\chi=-1$ and ${\rm Path}(\Gamma)=4$. The first equality follows from
\cite[Example 4.4.1]{nem.lattice},  or by using (\ref{eq:min}),
$\chi(\zmi)=-1$ and the existence of an analytic structure with $p_g=2$.
The second equality follows  again from \cite{nem.lattice} (see also the description of the $\chi$--function for graphs with two
nodes in \cite{laszlo}). Nevertheless,
%${\rm path}(\Gamma)= 4$
we will verify it  below as well.

 With these notations we prove the following.

\begin{thm}[Cf. \sref{s:pre}, \sref{s:<4}] For any analytic structure one has
$\coeff_{E_0}(\zma)\le 2$ (where $E_0$ is the $(-13)$-curve), and $p_g\X\le 3$.
If $\X$ is Gorenstein, then $p_g\X=3$.
\end{thm}

\begin{thm}
Any analytic structure satisfies one of the following properties:
\begin{enumerate}
\item $\zma=\zmi$, $p_g\X=3$, and $\X$ is a non-Gorenstein
 Kodaira singularity (cf. \thmref{t:ma=mi}).
\item $\zma=2\zmi$ and $\X$ is of splice type (hence Gorenstein
with $p_g\X=3$, cf. \thmref{t:splice}).
\item  $2\zmi\leq \zma < 3\zmi$ (there are three cases, see below),
$p_g\X=2$ and $\X$ is not Gorenstein (cf. \thmref{t:gen} and Section
    \ref{s:nez2z}).
\end{enumerate}
\end{thm}

\begin{cor}
The following are equivalent:
\begin{enumerate}
\item $\zma=\zmi$;
\item $\X$ is a Kodaira singularity.
\end{enumerate}
\end{cor}
%Examples ???????????? indicate that even this class of Kodaira singularities might have a finer stratification
%(decomposition). \marginpar{{\bf ADD}}
\begin{cor}
The following are equivalent:
\begin{enumerate}
\item $\zma=2\zmi$, $p_g\X=3$;
\item $\X$ is of splice type (complete intersection);
\item $\X$ is Gorenstein.
\end{enumerate}
\end{cor}
For  Kodaira  (or Kulikov) singularities see  \cite{kar.p,Stevens85},
for splice singularities see \cite{nw-HSL}.

\begin{rem}
(1) In general,
 a Gorenstein singularity with integral homology sphere link and with
 $\zmi^2=-1$ is not necessarily  of splice type.
An example can be found in
\cite[4.6]{supiso} (where the minimal good graph is even  star-shaped).

(2) For the two  cases with
$p_g=3$ (non--Gorenstein Kodaira and splice complete intersection) we provide precise realizations; however for the $p_g=2$ cases
we will not give the realizations (e.g. equations) in this article.

%\vspace{2mm}
(3)
The next table lists all the possible analytic structures supported by $\Gamma$
with some of their
key properties. $E$ is the exceptional curve
of the minimal resolution. For the notation $E_i^*$ see Section \ref{s:2}.
\[
\begin{array}{c|c|c|c|c|c|c}
\hline
\zma & \text{Gorenstein} & p_g & h^1(\cO_E(-E)) & h^1(\cO_E(-2E)) & \mult & \emb \\
\hline
\zmi & \text{No (Kodaira)}  & 3 & 1 & 0 & 3 & 4 \\
\hline
2\zmi & \text{Yes (splice)}& 3 & 0 & 1 & 4 & 4 \\
\hline
2\zmi \ \mbox{or} & \text{No} & 2 & 0 & 0 & 6 & 7 \\

E^*_1\ \mbox{or} \ E^*_4 &  &  &  &  &  &  \\
\hline
\end{array}
\]
%\marginpar{{\bf VERIFY}}

(4) In most of the proofs we use `computation sequences'.
Computation sequences  were introduced and deeply exploited
 by Laufer, they constitute a powerful machinery in the theory  of surface singularities. The present manuscript supports this fact as well.
\end{rem}
%\vspace{2mm}

\begin{rem} After we finished our manuscript the referee drew our attention to the excellent article \cite{Konno-yaucycles} of K. Konno, which we were not aware of. We thank the referee for this information. Indeed, our proofs and arguments
 and some of the statements have overlaps with the results of this article,
 which contains several important results regarding the key cycles of a resolution of a normal surface singularity.

After this information, however,  we decided not to change the structure (and the proofs) of our statements, in this way the present manuscript still remains
(more or less) self-contained and more readable. In this Remark we wish to
list some of the overlaps and give the credits to  \cite{Konno-yaucycles}.
(Definitely, this list covers only the overlaps, and not the huge
amount of results of  \cite{Konno-yaucycles}.)

 In  \cite{Konno-yaucycles} the author studies singularities with $Z_{min}^2=-1$. Our main example belongs to this family too,
 in fact, it even  belongs to the
 simplest class of `essentially irreducible $Z_{min}$' of Konno. For example,
 in `essentially irreducible $Z_{min}$' case, the fact that $p_g\leq 4$ when
 $Z_{min}^2=-1$ and $\chi(Z_{min})=\min \chi= -1$ is shown in Theorem 3.9
 of  \cite{Konno-yaucycles}. Furthermore, in
 \cite[Th. 3.9]{Konno-yaucycles} is also stated
 that the singularity must be a doublepoint whenever $p_g=4$. (This
 overlaps with the first part of our Theorem \ref{t:pg3}.)
 Also, the calculations of the present note in the Gorenstein
 case (\S\ref{s:2},I) is much similar to \cite[Th. 3.11]{Konno-yaucycles},
 which might even shorten slightly the proof of our Theorem \ref{t:splice}.
 A related statement can be found also in \cite[Lemma 3.4]{Konno-yaucycles}.
\end{rem}

 \begin{acknowledgement}
The second  author thanks
the R\'enyi Institute of Mathematics, Budapest, Hungary,
for the warm  hospitality during his visit.

\end{acknowledgement}

\section{Preliminary}\label{s:pre}
Let $\X$ be  a normal complex surface singularity and $\pi\:\tX\to X$ a resolution with exceptional set $E$.
 Let $\{E_i\}_{i\in \cI}$ denote the set of irreducible components of $E$.
We denote by $\Gamma$ the resolution graph of $\X$.
The group of cycles is defined by $L:=\sum _{i\in \cI}\Z E_i$.
Let us simplify into $DD'$ the intersection number $(D,D')$.
For any function $f\in H^0(\cO_{\tX})$, $f\ne 0$, let $(f)_E$ denote the exceptional part of $\di (f)$, namely, $(f)_E=\sum_{i\in \cI}\ord_{E_i}(f\circ \pi)E_i \in L$.
A divisor $D$ on $\tX$ is said to be nef (resp. anti-nef) if $DE_i\ge 0$ (resp. $DE_i\le 0$) for all $i\in \cI$.

We write $h^i( * )=\dim_{\C}H^i(*)$.
Moreover, for an effective cycle $l\in L$ we write
%\[
$H^i(l):=H^i(\cO_{l})$,
%\]
and  $\chi(l)$ denotes the Euler characteristic
$\chi(\cO_l)=h^0(l)-h^1(l)$.  %Then  $p_a(l)=1-\chi(l)$.
By Riemann--Roch formula, for a divisor $D$ on $\tX$,
\[
\chi(\cO_l(D))=h^0(\cO_l(D))-h^1(\cO_l(D))=\chi(l)+Dl=-(l^2-Z_Kl)/2+Dl,
\]
where $Z_K$ denotes the canonical cycle (see \defref{d:cycles}).
The expression $\chi(l) = -(l^2-Z_Kl)/2 $ is extended for any $l\in L$.

\begin{defn}\label{d:cycles}
We define the (minimal) Artin  cycle $\zmi$, the maximal ideal cycle $\zma$, and the cohomological cycle $\zco \in L$ as follows:
\begin{enumerate}
\item $\zmi=\min\defset{Z>0}{\text{$Z$ is anti-nef}}$.
\item $\zma=\min\defset{(f)_E}{f\in \m_{X,o}}$, where $\m_{X,o}$ is the maximal ideal of $\cO_{X,o}$.
\item $\zco=\min\defset{Z>0}{h^1(\cO_Z)=p_g\X}$ if $p_g\X>0$. $\zco=0$ if $p_g\X=0$.
\item
The canonical cycle $Z_K\in L\otimes \Q$ is defined by $K_{\tX}E_i=-Z_KE_i$ for all $i\in \cI$.
If $Z_K\in L$, then $\X$ or $\Gamma$ is said to be numerically Gorenstein.
\end{enumerate}\end{defn}

For the  existence of the unique cohomological cycle on any resolution (with the property
$h^1(Z)<p_g$ for any $Z\not\geq \zco$) see Reid \cite[\S 4.8]{chap}.
One has $\zco\leq \lfloor Z_K\rfloor$.

 Recall that $\X$ is Gorenstein if and only if $-Z_K\sim K_{\tX}$ (linear equivalence on $\tX$).

\begin{rem}
Let $k$ be a positive integer.

(1)
If $\zma=k\zmi$, $\tX$ is the minimal resolution, and $\cO_{\tX}(-\zma)$ has no base point, then the same equality holds on any resolution.

(2)
If $\zma=k\zmi$ on a resolution, then the same equality holds on the minimal resolution.
\end{rem}

\begin{thm}[{Konno \cite[\S 3]{Konno-coh}}]
\label{t:chc}
(1)  If $\X$ is Gorenstein  and  $p_g\X\ge 2$, then $p_g\X> p_a(\zmi)=1-\chi(\zmi)$.

(2) Assume that $\X$ is numerically Gorenstein and $Z_K\ge 0$. Then $\X$ is Gorenstein if and only if $Z_K=\zco$.
\end{thm}

%In fact, (\ref{eq:min}) and part (1) of Theorem \ref{t:chc} have the following immediate generalization:
% If $\X$ is Gorenstein  and  $p_g\X\ge 2$, then $p_g\X> 1-{\rm min}\chi$.

Next, assume that the link of $\X$ is a $\Q$-homology sphere and the graph $\Gamma$ is numerically Gorenstein.
 It is not hard to verify that
in the numerically Gorenstein case
%one can take $l_t=Z_K$ for the last cycle of the pathes
%in the definition of
${\rm Path}(\Gamma)={\rm Path}(Z_K)$
(a detailed proof can be found in \cite{Book}).
 %Accordingly, let $\cP_K$ be the set of pathes as in the
%introduction with the additional property $l_t=Z_K$, and set
%${\rm Path}(\Gamma):=\min_{\gamma\in\cP_K} S(\gamma)$.
The next results analyse certain   cases when
%those analytic structures for which
the inequality  $p_g\X\le {\rm Path}(\Gamma)$ from (\ref{eq:max}) is strict.

\begin{thm}
\label{t:bd}
Assume that $\Gamma$ is numerically Gorenstein and   $Z_K> \zco$ for some
analytic structure $(X,o)$ (that is, $(X,o)$ is not Gorenstein). Then, if
one of the following properties hold:

(1) either
%\begin{equation}\label{eq:surjectivity}
$\{\gamma\in \cP(Z_K)\,:\, S(\gamma)={\rm Path}(\Gamma)\}\to \{E_i\}_{i\in\cI}, \ \ \gamma\mapsto E_{i(t-1)}$, %\end{equation}
is surjective, or

(2) the  support $|Z_K-\zco|$ is $E$,

\noindent  then $p_g(X,o)<{\rm Path}(\Gamma)$.
\end{thm}
\begin{proof} We prove that if $p_g={\rm Path}(\Gamma)$ and
the surjectivity (1) holds then $\zco=Z_K$. Indeed,
 the assumption $p_g={\rm Path}(\Gamma)$ implies that along a path
 (any path) $\gamma$ with $p_g={\rm Path}(\Gamma)
 =S(\gamma) $,
 whenever $p_g$ can grow with $E_{i(k)}l_k-1>0$, it necessarily grows with this amount.
 On the other hand,
%if $\Gamma$ is numerically Gorenstein,
for any choice of $\gamma$, $l_{t-1}$ has the form
$Z_K-E_{i(t-1)}$. Since $l_{t-1} E_{i(t-1)}-1=
2\chi(E_{i(t-1)})-1=1$, the assumption $p_g={\rm Path}(\Gamma)$ implies that
necessarily   $h^1(Z_K-E_{i(t-1)})<h^1(Z_K)=p_g$. By the surjectivity (1)
we get that
this must be the case for any $E_i$, that is, $h^1(Z_K-E_{i})<h^1(Z_K)=p_g$ for any $i\in\cI$.
This shows that $\zco=Z_K$.
%, hence $(X,o)$ is Gorenstein by Theorem \ref{t:chc}.

Suppose that the condition (2) holds. Fix
 $\gamma\in \cP(Z_K)$, $\gamma=\{l_k\}_{k=0}^t$,
with $S(\gamma)={\rm Path}(\Gamma)$. Let
$\gamma'$ be the shorter path $\gamma'=\{l_k\}_{k=0}^{t-1}$. Then by similar  computation as above $S(\gamma')=S(\gamma)-1$. Hence,
by (\ref{eq:maxZ}),
$p_g=h^1(\zco)\leq
h^1(\cO_{Z_K-E_{i(t-1)}})\leq S(\gamma')<S(\gamma)={\rm Path}(\Gamma)$.
\end{proof}
\begin{ass}
From now on, we assume that the minimal good resolution graph $\Gamma$ of $\X$ is as in \figref{fig:gamma}.
\end{ass}
The cycles $\zmi$ and $Z_K$ are shown in the next picture:

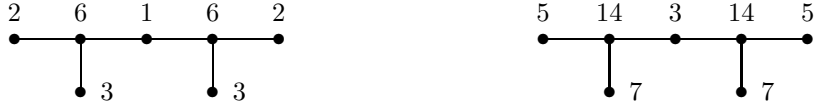
\begin{picture}(300,45)(10,0)
\put(25,25){\circle*{4}}
\put(50,25){\circle*{4}}
\put(75,25){\circle*{4}}
\put(100,25){\circle*{4}}
\put(125,25){\circle*{4}}
\put(50,5){\circle*{4}}
\put(100,5){\circle*{4}}
\put(25,25){\line(1,0){100}}
\put(50,25){\line(0,-1){20}}
\put(100,25){\line(0,-1){20}}
\put(25,35){\makebox(0,0){$2$}}
\put(50,35){\makebox(0,0){$6$}}
\put(75,35){\makebox(0,0){$1$}}
\put(100,35){\makebox(0,0){$6$}}
\put(125,35){\makebox(0,0){$2$}}
\put(60,5){\makebox(0,0){$3$}}
\put(110,5){\makebox(0,0){$3$}}

\put(225,25){\circle*{4}}
\put(250,25){\circle*{4}}
\put(275,25){\circle*{4}}
\put(300,25){\circle*{4}}
\put(325,25){\circle*{4}}
\put(250,5){\circle*{4}}
\put(300,5){\circle*{4}}
\put(225,25){\line(1,0){100}}
\put(250,25){\line(0,-1){20}}
\put(300,25){\line(0,-1){20}}
\put(225,35){\makebox(0,0){$5$}}
\put(250,35){\makebox(0,0){$14$}}
\put(275,35){\makebox(0,0){$3$}}
\put(300,35){\makebox(0,0){$14$}}
\put(325,35){\makebox(0,0){$5$}}
\put(260,5){\makebox(0,0){$7$}}
\put(310,5){\makebox(0,0){$7$}}
\end{picture}

One easily verifies that $\chi(\zmi)=-1$, hence $h^1(\zmi)=2$, which implies $p_g\geq 2$.
(In fact, ${\rm min}\chi$ is also $-1$, cf. \cite[4.4.1]{nem.lattice}.)

For any path $\gamma=\{l_k\}_k$ we say that $\gamma$ has a simple jump at $k$ if $E_{i(k)}l_k=2$. %\marginpar{As definition, $E_{i(k)}l_k\ge 2$ ???}

Let us prove first that for the above graph one has ${\rm Path}(\Gamma)\leq 4$. For this we have to construct a path
with (at most) four simple jumps.
%\marginpar{For \thmref{t:pg<path}, ${\rm Path}(\Gamma)= 4$ should be proved.}

We start with $l_0=0$, then we add a base-element, say the $(-13)$--vertex $E_0$.
Then there exists a `Laufer computation sequence' starting from $E_0$ and ending with $\zmi$, determined by
Laufer's algorithm (for the Artin cycle) \cite{la.ra}, which has exactly two simple jumps, and at all the other steps
$E_{i(k)}l_k=1$. Next, we add a base--element   (say $E_5$, one of the $(-1)$--base cycles) to $\zmi$. Then, again,
there is a computation sequence starting with  $\zmi+E_5$ and ending with $2\zmi$ with exactly one simple jump and
at all the other steps $E_{i(k)}l_k=1$. Finally, constructed in similar way,
there is a increasing sequence starting with $2\zmi$ and ending with $Z_K$ such that there are two
steps with $E_{i(k)}l_k=0$ (including the very first one), one simple jump, and at  all the other steps $E_{i(k)}l_k=1$.
(Since $\chi(Z_K-E_i)=1>\chi(Z_K)=0$, a jump necessarily must appear.)

This shows that ${\rm Path}(\Gamma)\leq 4$, hence for any analytic structure $p_g\leq 4$.
%If there exists an analytic structure with $p_g=4$
%then  ${\rm Path}(\Gamma)=4$ by (\ref{eq:min})
% (in fact, in Section \ref{s:<4}
%we prove that such an analytic structure does not exist).

In Section \ref{s:Kod} we show (using also from
Section \ref{s:<4} that $p_g<4$) that the Kodaira analytic structure satisfies $p_g=3$ and $\zco \le 2\zmi\leq Z_K-E$ (cf. (\ref{eq:zcoh})). Hence, by Theorem \ref{t:bd},
 ${\rm Path}(\Gamma)=4$.

Moreover, analysing the long exact cohomological sequences at each step along the pathes considered above,
we obtain  that
\begin{equation}\label{eq:h1s}
\left\{
\begin{array}{l}
h^1(\zmi)=2, \\
h^1(2\zmi)\leq h^1(\zmi)+1,\\
h^1(Z_K)=p_g\leq h^1(2\zmi)+1.\end{array}\right.
\end{equation}
 Furthermore, the reader is invited to verify (by constructing the corresponding pathes) that the above sequence--construction procedure
 has the following additional property as well.
For any $i\in\cI$, there is a sequence  starting with $2\zmi$ and ending with $Z_K$, with
all the properties listed above, and which ends with $E_{i}$ (that is, at the very last step we have to
add $E_i$). Therefore, Theorem \ref{t:bd} and (\ref{eq:h1s}) read as follows.

\begin{cor}\label{cor:Gorenstein}
If there exists a singularity $(X,o)$
with graph  $\Gamma$ (as in Figure 1) and $p_g=4$ then $(X,o)$ should be Gorenstein and necessarily  $h^1(m\zmi)=m+1$ for $m=1,2,3$.
(Note that $3\zmi \ge Z_K$.)
\end{cor}
This will be an important ingredient in proving that $p_g=4$ is not realized.
%\vspace{2mm}

In the rest of this section, we assume that $\pi\: \tX\to X$ is the {\it minimal} resolution.
Then $E$ is an irreducible curve with $E^2=-1$ and it has two ordinary cusps;
it corresponds to the $(-13)$--curve in \figref{fig:gamma}.
One verifies the following facts.
\begin{equation}\label{eq:chiE}
h^1(E)=2, \quad
\chi(\cO_E(-nE))=n-1,
\quad \chi(nE)=(n^2-3n)/2 \; \text{ for } n\ge 0.
\end{equation}
 From the exact sequence
\begin{equation}\label{eq:0-1}
0 \to \cO_{\tX}(-E)\to \cO_{\tX}\to \cO_E \to 0,
\end{equation}
we have
\begin{equation}\label{eq:pg-2}
h^1(\cO_{\tX}(-E))=p_g\X-2.
\end{equation}

By adjunction formula, we obtain that $Z_K=3E$.

By the Grauert-Riemenschneider vanishing theorem, $H^1(\cO_{\tX}(-3E))=0$. Therefore, the exact sequence
$
0\to \cO_{\tX}(-3E)\to \cO_{\tX}(-2E)\to \cO_{E}(-2E)\to 0,
$
implies
\begin{equation}\label{eq:2E}\left\{
\begin{array}{ll}
 (a) \ \ \dim \frac{H^0(\cO_{\tX}(-2E))}{H^0(\cO_{\tX}(-3E))}=\dim  H^0(\cO_E(-2E))\geq \chi(\cO_E(-2E))=1,\\ \\
 (b) \ \ \ h^1(\cO_{\tX}(-2E))=h^1(\cO_E(-2E)).\end{array}\right.
\end{equation}
Hence, the definition of $\zma$ and  (\ref{eq:2E})(a) imply the following.
\begin{prop}\label{p:zle2e}
$\zma\le 2E$ on the minimal resolution.
\end{prop}

\section{A singularity with $p_g\ge4$ does not exist}
\label{s:<4}
The aim of this section is to prove the following.
\begin{thm}\label{t:pg3}
For all analytic structures $(X,o)$  supported on $\Gamma$ one has
 $2\le p_g\X\le 3$. If $\X$ is Gorenstein, then $p_g\X=3$.
\end{thm}

The proof consists of several step. Notice that the second part follows from \eqref{eq:min} and
Theorem \ref{t:chc}, since $1-\chi(\zmi)=2$ (provided that we verify that $p_g\leq 3$).

Hence we need to prove that $p_g=4$ cannot occur.
To do this, we
 assume that $p_g\X=4$ for certain $(X,o)$ and  we will deduce a contradiction.

By Corollary \ref{cor:Gorenstein} $(X,o)$ is necessarily Gorenstein.

Let $\tX$ be the minimal resolution. Then $K_{\tX}=-Z_K=-3E$.

Moreover, by  Corollary \ref{cor:Gorenstein} again, in the minimal good resolution
$h^1(m\zmi)=m+1$ for $m=1,2,3$. Hence in the minimal resolution
(e.g.  by Leray spectral sequence argument)
\begin{equation}\label{eq:h1s2}
h^1(mE)=m+1 \ \ (m=1,2,3).
\end{equation}
%Hence $\zco=3E=Z_K$.

From \eqref{eq:pg-2}  $h^1(\cO_{\tX}(-E))=2$, and from \eqref{eq:2E} we also have $h^1(\cO_{\tX}(-2E))=1$, since
 $h^1(\cO_E(-2E))=h^0(\cO_E)=1$ by duality.
From the exact sequence
\begin{equation}\label{eq:E2E}
0 \to \cO_E(-E)\to \cO_{2E}\to \cO_E \to 0,
\end{equation}
we also obtain  $h^1(\cO_E(-E))=1$. So $h^0(\cO_E(-E))=1$ since $\chi(\cO_E(-E))=0$.

Since $h^1(\cO_{\tX}(-2E))-h^1(\cO_{\tX}(-E))+h^1(\cO_E(-E))=0$,
from the exact sequence
\begin{equation}\label{eq:1-2}
0 \to \cO_{\tX}(-2E)\to \cO_{\tX}(-E)\to \cO_E(-E) \to 0,
\end{equation}
$H^0(\cO_{\tX}(-E))\to H^0(\cO_E(-E))\cong \C$ is surjective.
Therefore, $\zma=E$.
Let $s\in H^0(\cO_E(-E))$ be the image of a general function $f\in H^0(\cO_{\tX}(-E))$.
Consider the exact sequence
\[
0\to \cO_E(-E) \xrightarrow{\times s} \cO_E(-2E) \to \cO_P(-2E)\to 0,
\]
where $P\in E$ is the zero of $s$. Since $\deg \cO_E(-E)=1$, $P$ is a nonsingular point of $E$.
Since  $h^1(\cO_E(-2E))=1=h^1(\cO_E(-E))$ we get that
$H^0(\cO_E(-2E))\to H^0(\cO_P(-2E))$ is surjective, hence $P$ is not a base point of
$H^0(\cO_E(-2E))$.
Furthermore,  since $H^0(\cO_{\tX}(-2E))\xrightarrow{r} H^0(\cO_E(-2E))$ is surjective, there exists a function $g\in H^0(\cO_{\tX}(-2E))$ such that $r(g)(P)\ne 0$ and $(g)_E=2E$.
We can choose local coordinates $x,y$ at $P$ such that $E=\{x=0\}$, $f=xy$, $g=x^2$.
Then $\m_{X,o}\cO_{\tX}=(x,y)\cO_{\tX}(-E)$ at $P$, or,  $\m_{X,o} \cO_{\tX}=\m_P\cO_{\tX}(-E)$. Hence $\mult\X=-E^2+1=2$.

Now, it is well--known that a normal surface singularity with multiplicity two is necessarily a hypersurface of suspension type:  $\X=(\{z^2+h(x,y)=0\}, o)$ in  suitable local coordinates.

However, this is impossible by the following proposition
and by the fact  that the splice diagram of $\Gamma$ is

\begin{picture}(300,40)(50,0)
\put(125,25){\circle*{4}}
\put(150,25){\circle*{4}}
%\put(175,25){\circle*{4}}
\put(200,25){\circle*{4}}
\put(225,25){\circle*{4}}
\put(150,5){\circle*{4}}
\put(200,5){\circle*{4}}
\put(125,25){\line(1,0){100}}
\put(150,25){\line(0,-1){20}}
\put(200,25){\line(0,-1){20}}
\put(145,30){\makebox(0,0){$3$}}
\put(158,30){\makebox(0,0){$7$}}
\put(193,30){\makebox(0,0){$7$}}
\put(206,30){\makebox(0,0){$3$}}
\put(155,17){\makebox(0,0){$2$}}
\put(205,17){\makebox(0,0){$2$}}
\end{picture}

\begin{prop} \cite{NWCasson} \
Assume that the link of $\{z^n+h(x,y)=0\}$ is an integral homology sphere. Then the following facts hold.
\begin{enumerate}
\item $h$ is irreducible;

\item Assume that the splice diagram of $h$ is the following (for details see \cite{EN}):

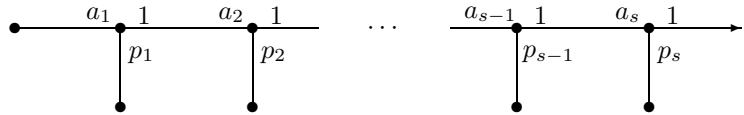
\begin{picture}(400,55)(40,20)
\put(60,60){\circle*{4}}
\put(100,60){\circle*{4}}
\put(150,60){\circle*{4}}
\put(250,60){\circle*{4}}
\put(300,60){\circle*{4}}
\put(100,30){\circle*{4}}
\put(150,30){\circle*{4}}
\put(250,30){\circle*{4}}
\put(300,30){\circle*{4}}
%\put(350,60){\makebox(0,0){$\tilde{C}$}}
\put(92,65){\makebox(0,0){$a_1$}}
\put(142,65){\makebox(0,0){$a_2$}}
\put(240,65){\makebox(0,0){$a_{s-1}$}}
\put(292,65){\makebox(0,0){$a_s$}}

\put(109,65){\makebox(0,0){$1$}}
\put(159,65){\makebox(0,0){$1$}}
\put(259,65){\makebox(0,0){$1$}}
\put(309,65){\makebox(0,0){$1$}}

\put(108,50){\makebox(0,0){$p_1$}}
\put(158,50){\makebox(0,0){$p_2$}}
\put(262,50){\makebox(0,0){$p_{s-1}$}}
\put(308,50){\makebox(0,0){$p_s$}}
\put(200,60){\makebox(0,0){$\cdots$}}
%\put(50,60){\framebox(125,0){}}
%\put(225,60){\framebox(75,0){}}
\put(100,30){\line(0,1){30}}
\put(150,30){\line(0,1){30}}
\put(250,30){\line(0,1){30}}
\put(300,30){\line(0,1){30}}
%\put(150,20){\framebox(0,40){}}
%\put(250,20){\framebox(0,40){}}
%\put(300,20){\framebox(0,40){}}
%\put(50,60){\vector(1,0){280}}
\put(60,60){\line(1,0){115}}
\put(225,60){\vector(1,0){110}}%\put(335,60){\circle*{4}}
\end{picture}

\noindent Then $(a_ip_i,n)=1$ for all $i$.

\item The splice diagram of $\{z^n+h(x,y)=0\}$ is

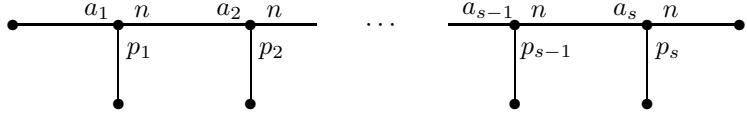
\begin{picture}(400,55)(40,20)
\put(60,60){\circle*{4}}
\put(100,60){\circle*{4}}
\put(150,60){\circle*{4}}
\put(250,60){\circle*{4}}
\put(300,60){\circle*{4}}
\put(100,30){\circle*{4}}
\put(150,30){\circle*{4}}
\put(250,30){\circle*{4}}
\put(300,30){\circle*{4}}
%\put(350,60){\makebox(0,0){$\tilde{C}$}}
\put(92,65){\makebox(0,0){$a_1$}}
\put(142,65){\makebox(0,0){$a_2$}}
\put(240,65){\makebox(0,0){$a_{s-1}$}}
\put(292,65){\makebox(0,0){$a_s$}}

\put(109,65){\makebox(0,0){$n$}}
\put(159,65){\makebox(0,0){$n$}}
\put(259,65){\makebox(0,0){$n$}}
\put(309,65){\makebox(0,0){$n$}}

\put(108,50){\makebox(0,0){$p_1$}}
\put(158,50){\makebox(0,0){$p_2$}}
\put(262,50){\makebox(0,0){$p_{s-1}$}}
\put(308,50){\makebox(0,0){$p_s$}}
\put(200,60){\makebox(0,0){$\cdots$}}
%\put(50,60){\framebox(125,0){}}
%\put(225,60){\framebox(75,0){}}
\put(100,30){\line(0,1){30}}
\put(150,30){\line(0,1){30}}
\put(250,30){\line(0,1){30}}
\put(300,30){\line(0,1){30}}
%\put(150,20){\framebox(0,40){}}
%\put(250,20){\framebox(0,40){}}
%\put(300,20){\framebox(0,40){}}
%\put(50,60){\vector(1,0){280}}
\put(60,60){\line(1,0){115}}
\put(225,60){\line(1,0){110}}\put(335,60){\circle*{4}}
\end{picture}
\end{enumerate}
\end{prop}

\section{The case $\zma=\zmi$}\label{s:Kod}

Proposition 2.7 and Lemma 2.9.1 of \cite[\S 2]{kar.p}
guarantee the existence of  a normal complex surface singularity $\X$ with minimal good resolution graph $\Gamma$
on which   $\zma=\zmi$.
Indeed, let us construct an `extended' graph $\Gamma^e$ by gluing  a $(-1)$--vertex
to the $(-13)$--vertex of $\Gamma$ by a new edge.
In this way we get a negative semi--definite graph. By a theorem
of Winters  \cite{Winters} there exists a family of projective curves $h_W:W\to (\C,0)$ such that $W$ is smooth,
the central
fiber is encoded by $\Gamma^e$,  and the nearby fibers are smooth. Let $\tX$ be a convenient small
neighbourhood of the union of central curves indexed by $\Gamma$. Then this union of curves can be contracted by
Grauert theorem \cite{GRa} to get a singularity $(X,o)$ and  $\tX$ serves as its  minimal good resolution,
on which the restriction $h$ of $h_W$ is a function with $(h)_E=\zmi$.

 An analytic type constructed in this way is called Kodaira \cite{kar.p} (or Kulikov \cite{Stevens85}).

We shall prove the following.
\begin{thm}\label{t:ma=mi}
If $\zma=\zmi$ on the minimal good resolution, then $\X$ necessarily is a non-Gorenstein Kodaira singularity with $p_g\X=3$, $\emb\X=4$ and $\mult\X=3$. Furthermore such $\X$ is the total space of a one-parameter family of the curve singularity defined by
$\rank\begin{pmatrix}
z_1 & z_2 &z_3 \\ z_2 &z_3 & z_1^2
\end{pmatrix}<2$ in $(\C^3,0)$.
\end{thm}
\begin{proof}
We note that $\zma=E$ on the minimal resolution if and only if $\zma=\zmi$ on the minimal good resolution, because if $\di(f)=E+H$ on the minimal resolution, then $H$ intersects $E$ transversally.

Assume that $\tX$ is the minimal resolution and that $\zma=E$.
Note that $H^1(\cO_{\tX}(-nE))=0$ for $n\ge 3$ by the vanishing theorem (cf. \cite{Gi}).
Then $\X$ is a Kodaira singularity by \cite[2.9.1]{kar.p} and $\cO_{\tX}(-E)$ has no fixed component. Hence $\dim H^0(\cO_{\tX}(-E))/H^0(\cO_{\tX}(-2E))\ne 0$.
From the exact sequence \eqref{eq:1-2}, we have
\begin{multline*}
h^1(\cO_{\tX}(-E)) \ge h^1(\cO_E(-E))=h^0(\cO_E(-E))\\
\ge \dim_{\C} H^0(\cO_{\tX}(-E))/H^0(\cO_{\tX}(-2E)) \ge 1.
\end{multline*}
Since by \thmref{t:pg3} $p_g\X\le 3$, in fact we have  $p_g\X=3$ by \eqref{eq:pg-2},  and all the
inequalities above are equalities. Hence, via (\ref{eq:E2E}),
\begin{equation}\label{eq:zcoh}
\zco=2E.\end{equation}
By \thmref{t:chc}, $\X$ is not Gorenstein.
Since $H^1(\cO_{\tX}(-3E))=0$, it follows from \cite[3.1]{OWYrees} (cf. also with the exact sequence from (\ref{eq:1-2}))
that $1=h^1(\cO_{\tX}(-E))>h^1(\cO_{\tX}(-nE))$ for $n\ge 2$.
 In particular, $h^1(\cO_{\tX}(-nE))=0$ for $n\ge 2$.  Thus we obtain that
$H^0(\cO_{\tX}(-nE)) \to H^0(\cO_E(-nE))$ is surjective for $n\ge 0$ and
$h^0(\cO_E(-nE))=n-1$ for $n\ge 2$ by \eqref{eq:chiE}.

Let us compute the multiplicity of $\X$.
Since $h^0(\cO_E(-E))=h^0(\cO_E(-2E))=1$, $\cO_{\tX}(-E)$ and $\cO_{\tX}(-2E)$ have a base point $P$.
Take a general section $s\in H^0(\cO_E(-E))$, % which is the image of $f\in H^0(\cO_{\tX}(-E))$
and consider the exact sequence
\[
0 \to \cO_E(-2E))\xrightarrow{\times s} \cO_E(-3E)
\to \cO_P(-3E) \to 0.
\]
Then $H^0(\cO_E(-3E)) \to H^0(\cO_P(-3E))$ is  surjective  since $h^0(\cO_E(-2E))=1$ and
 $h^0(\cO_E(-3E))=2$.
Since $H^0(\cO_{\tX}(-3E))\xrightarrow{r} H^0(\cO_E(-3E))$ is surjective, $\cO_{\tX}(-3E)$ has no base point.
Hence a general function $g\in H^0(\cO_{\tX}(-3E))$ satisfies $r(g)(P)\ne 0$ and $(g)_E=3E$. As in \sref{s:<4}, for  suitable coordinates $x,y$ at $P$,
$\m_{X,o} \cO_{\tX}=(y,x^2)\cO_{\tX}(-E)$, where $E=\{x=0\}$.
Taking the blowing up $\phi_1\:X_1\to \tX$ at the base point $P$, we have a new base point $Q\in X_1$ such that $\m_{X,o}\cO_{X_1}=\m_Q\cO_{X_1}$. Let $\phi_2\:X_2\to X_1$ be the blowing up at the base point $Q$.
Let $E_i\subset X_i$ be the exceptional set of $\phi_i$, $Z_1=\phi_1^*E+E_1$, and $Z_2=\phi_1^*Z_1+E_2$.
Then the maximal ideal cycle on $X_2$ is $Z_2$ and $\cO_{X_2}(-Z_2)$ has no base point.
Hence $\mult \X=-Z_2^2=3$.
Since $\emb\X\le \mult\X+1=4$ (cf. \cite{Abhy}), and  $\X$ is not Gorenstein, we have $\emb\X=4$, because any hypersurface is Gorenstein.

%We show that $\X$ is not Gorenstein.
%Consider the following exact sequences.
%\begin{align*}
%&0 \to \cO_{X_1}(-Z_1) \to \cO_{X_1}(-\phi_1^*E) \to \cO_{E_1} \to 0,
%\\
%&0 \to \cO_{X_2}(-Z_2) \to \cO_{X_2}(-\phi_2^*Z_1) \to \cO_{E_2} \to 0.
%\end{align*}
%From the corresponding spectral sequences, we obtain that
%\[
%h^1(\cO_{X_1}(-\phi_1^*E))=h^1(\cO_{\tX}(-E)), \quad h^1(\cO_{X_2}(-\phi_2^*Z_1))=h^1(\cO_{X_1}(-Z_1)).
%\]
%Since $P$ and  $Q$ are  base points, we obtain
%\[
%h^1(\cO_{X_2}(-Z_2))=h^1(\cO_{X_1}(-Z_1))+1=h^1(\cO_{\tX}(-E))+1+1=3.
%\]
%Therefore $Z_2$ is the $p_g$-cycle in the sense of  \cite{OWYgood}.
%If $\X$ would be  Gorenstein, then one also would have $\zco=Z_K$ on $X_2$.
%Since $\supp(Z_K)=\supp(Z_2)$ and $Z_K\ge 0$, we have $Z_2\zco\ne 0$; this contradicts the characterization of the $p_g$-cycle (\cite[3.10]{OWYgood}).

Let $h\in \m_{X,o}$ be a general function.
Then
\[
\mult(\{h=0\},o)=\mult\X, \quad \emb(\{h=0\},o)=\emb\X-1.
\]
By the formula of Morales \cite[2.1.4]{mo.rr},
\[
\delta((\{h=0\},o))=-(Z_KZ_2+Z_2^2)/2=2=\emb((\{h=0\},o)) -1.
\]
Hence $(\{h=0\},o)$ is a partition curve $Y(3)$ in \cite[\S 3]{b-c.rat}.

This ends the proof of the theorem.
\end{proof}

\begin{ex}
We give defining equations of a Kodaira singularity with graph $\Gamma$.
Let us recall \cite[Example 6.3]{o.numGell}.
Let $(X',o)\subset (\C^4,o)$ be a singularity defined by
$$
\rank
\begin{pmatrix}
 x & y & z \\ y-3w^2  & z+w^{3} & x^2+6wy-2w^3
\end{pmatrix}<2.
$$
It is a numerically Gorenstein elliptic singularity. It shares the topological type the hypersurface singularity $(Y_{2},o):=\{x^2+y^3+z^{13}=0\}\subset (\C^3,o)$ with  $p_g(Y_2,o)=2$,
however $p_g(X',o)=1$.
The exceptional set $E'$ of the minimal resolution of $(X',o)$ consists of two rational curve $E'_1$ and $E_2'$ with $E_1'^2=-1$, $E_2'^2=-2$, $E_1'E_2'=1$ and $E_1'$ has an ordinary cusp. The maximal ideal cycle is $2E'_1+E_2'$.
The affine piece $V_1\subset \C^5$ of the partial resolution (see \cite[Example 6.3]{o.numGell}) of $(X',o)$ is defined by
 the equations
$$
sx= y-3w^2, \quad   sy= z+w^{3}, \quad  sz=x^2+6wy-2w^3.
$$
Consider the order of the coordinate functions on the exceptional set $E'$ on $V_1$.
Then  the order of $s$ is zero, and the order of $w$ is less than those of $x,y,z$.
Hence $\zma=(w)_{E'}$. Note that $H:=\di(w)-(w)_{E'}$ intersects $E_1'\setminus E_2'$ transversally.
The graph of $\di (w)$ on the minimal good resolution is as follows (the arrow corresponds to the strict transform of $H$):

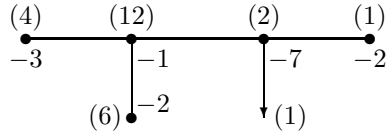
\begin{picture}(200,60)(-10,20)
\put(60,60){\circle*{4}}
\put(100,60){\circle*{4}}
\put(150,60){\circle*{4}}
\put(100,30){\circle*{4}}
\put(190,60){\circle*{4}}
\put(60,68){\makebox(0,0){$(4)$}}
\put(60,52){\makebox(0,0){$-3$}}
\put(150,68){\makebox(0,0){$(2)$}}

\put(100,68){\makebox(0,0){$(12)$}}
\put(90,30){\makebox(0,0){$(6)$}}
\put(160,30){\makebox(0,0){$(1)$}}
\put(190,68){\makebox(0,0){$(1)$}}

\put(108,52){\makebox(0,0){$-1$}}
\put(108,35){\makebox(0,0){$-2$}}
\put(158,52){\makebox(0,0){$-7$}}
\put(190,52){\makebox(0,0){$-2$}}
\put(100,30){\line(0,1){30}}
\put(150,30){\line(0,1){30}}
\put(150,60){\vector(0,-1){30}}
\put(60,60){\line(1,0){130}}
\end{picture}

Let $\phi\: \X \to (X',o)$ be the double cover of $X'$ brabched along $w=0$, namely, $\cO_{X,o}=\cO_{X',o}\{t\}/(t^2-w)$.
Then  $\X$ is defined by
$$
\rank
\begin{pmatrix}
 x & y & z \\ y-3t^4  & z+t^6 & x^2+6t^2y-2t^6
\end{pmatrix}<2.
$$
By the method of \cite[III. Appendix 1]{nem.5}, $\X$ has the resolution graph $\Gamma$, and $(t)_E=\zma=\zmi$.
\end{ex}

\section{The case $\zma=2\zmi$}\label{s:2}

Assume that $\tX$ is the minimal good resolution and $\zma=2\zmi$ on $\tX$. We express the irreducible components of  $E$ as $E_0, \dots, E_6$ as below.   %in \figref{fig:Ei}.

%\begin{figure}

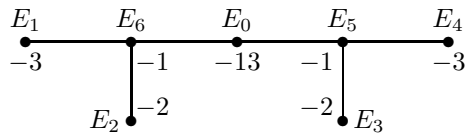
\begin{picture}(200,60)(0,20)
\put(60,60){\circle*{4}}
\put(100,60){\circle*{4}}
\put(140,60){\circle*{4}}
\put(180,60){\circle*{4}}
\put(100,30){\circle*{4}}
\put(180,30){\circle*{4}}
\put(220,60){\circle*{4}}
\put(60,68){\makebox(0,0){$E_1$}}
\put(60,52){\makebox(0,0){$-3$}}
\put(140,68){\makebox(0,0){$E_0$}}
\put(140,52){\makebox(0,0){$-13$}}
\put(220,52){\makebox(0,0){$-3$}}
\put(180,68){\makebox(0,0){$E_5$}}

\put(100,68){\makebox(0,0){$E_6$}}
\put(90,30){\makebox(0,0){$E_2$}}
\put(190,30){\makebox(0,0){$E_3$}}
\put(220,68){\makebox(0,0){$E_4$}}

\put(108,52){\makebox(0,0){$-1$}}
\put(108,35){\makebox(0,0){$-2$}}
\put(170,52){\makebox(0,0){$-1$}}
\put(170,35){\makebox(0,0){$-2$}}
\put(100,30){\line(0,1){30}}
\put(180,30){\line(0,1){30}}
\put(60,60){\line(1,0){160}}
\end{picture}
%\caption{The exceptional set of the minimal good resolution}
%\label{fig:Ei}
%\end{figure}

The cycle $E_i^*\in L$ is defined by $E_i^*E_i=-1$, $E_i^*E_j=0$ for all $j\ne i$. (In general, $E_i^*$ is an element of $L\otimes \Q$. In our case, $E_i^*\in L$ since the intersection matrix is unimodular.)
E.g.,  $\zmi=E_0^*$.
From the exact sequence
\[
0\to \cO_{\tX}(-2\zmi)\to \cO_{\tX}(-\zmi)\to \cO_{\zmi}(-\zmi)\to 0
\]
we have
\begin{equation}\label{eq:2-1}
h^1(\cO_{\tX}(-2\zmi))-h^1(\cO_{\tX}(-\zmi))=\chi(\cO_{\zmi}(-\zmi))=0.
\end{equation}
Note that this equality holds whenever $\zma\ge 2\zmi$.
\vspace{2mm}

\noindent {\bf I. \ The Gorenstein case.} \
%We prove the following.

\begin{thm}\label{t:splice}
Assume that $\zma=2\zmi$ on the minimal good resolution and $\X$ is Gorenstein. Then $\X$ is of splice type and the ``leading form''of the splice diagram equations are given by
\[
z_1^2z_2+z_3^2+z_4^3, \quad z_1^3+z_2^2+z_4^2z_3,
\]
where $z_i$ corresponds to the end $E_i$.
Furthermore, we have $\mult\X=4$ and that $\cO_{\tX}(-\zma)$ has no base points.
\end{thm}

The graph $\Gamma$ satisfies the semigroup condition and we read the above defining equations
 from \cite{nw-HSL}.
If $X$ is of splice type, we have $\mult\X=2\cdot 2=4$, because the tangent cone is defined by the regular sequence $z_3^2$, $z_2^2$. Furthermore,
 $\cO_{\tX}(-\zma)$ has no base points since $-\zma^2=4$ (or, by
 analysing the divisors $E_1^*$ and $E_4^*$ of $z_1$ and $z_4$).
Therefore,  it is sufficient to prove that the end curve condition is satisfied (see \cite{nw-ECT}).

Since $\X$ is Gorenstein, we have $p_g\X=3$ by \thmref{t:pg3}.
Therefore, from  \eqref{eq:pg-2} and \eqref{eq:2-1},
\begin{equation}\label{eq:44}h^1(\cO_{\tX}(-\zmi))=h^1(\cO_{\tX}(-2\zmi))=1.\end{equation}

\begin{lem}\label{l:e4}
Let $Z=E_4^*$.
Then $\cO_{\tX}(-Z)$ has no fixed component.
In particular, there exists a function $f\in H^0(\cO_{\tX}(-Z))$ such that $\di (f)=Z+H$, where $H$ is non-exceptional and $HE=HE_4=1$ (that is, $H$ is a `cut' of $E_4$), and hence the end curve condition at $E_4$ is satisfied.
\end{lem}
\begin{proof}
If $\cO_{\tX}(-Z)$ has a fixed component, then every component of $E-E_4$ is also a fixed component because for any cycle $D>0$ and the minimal anti-nef cycle $D'$ such that $D'\ge D$, we have $H^0(\cO_{\tX}(-D'))=
H^0(\cO_{\tX}(-D))$ (and if $D'>Z$ then $D'\geq Z+E$ too). We will show that $E_6$ cannot be a fixed component.

Since $Z>\zmi$ (hence $h^1(\cO_Z)\geq h^1(\cO_{\zmi})=2$), $\zco=Z_K$ and
$
C:=Z_K-Z=E_0+E_1+E_2+2E_6>0,
$
we obtain %it follows from \thmref{t:chc}
that $h^1(\cO_Z)=2$.
 Thus
\begin{equation}\label{eq:11}
h^1(\cO_{\tX}(-Z))\ge p_g\X-h^1(\cO_Z)=1.
\end{equation}
Consider the exact sequences
\begin{align*}
&0\to \cO_{\tX}(-Z-(C-E_6))\to \cO_{\tX}(-Z-E_6)\to
\cO_{C-2E_6}(-E_6)\to 0,\\
&0\to \cO_{\tX}(-Z-C)\to \cO_{\tX}(-Z-(C-E_6))\to
\cO_{E_6}(-(C-E_6))\to 0.
\end{align*}
Since $h^1(\cO_{C-2E_6}(-E_6))=h^1(\cO_{\tX}(-Z-C))=0$ and $h^1(\cO_{E_6}(-(C-E_6)))=1$, we obtain
\begin{equation}\label{eq:22}
1\ge h^1(\cO_{\tX}(-Z-E_6)).
\end{equation}
Therefore,  (\ref{eq:11}) and (\ref{eq:22}) implies that $h^1(\cO_{\tX}(-Z))\ge h^1(\cO_{\tX}(-Z-E_6))$.

\noindent This fact, and the  exact sequence
\[
0\to \cO_{\tX}(-Z-E_6)\to \cO_{\tX}(-Z)\to \cO_{E_6}\to 0
\]
show that
  the restriction map $H^0(\cO_{\tX}(-Z)) \to H^0(\cO_{E_6})$ is non-trivial. Hence $E_6$ cannot be a fixed component.
\end{proof}

\begin{lem}\label{l:e3*}
Let $Z=E_3^*$.
Then $\cO_{\tX}(-Z)$ has no fixed component.
\end{lem}

\begin{proof} Similarly  as in the proof of the previous lemma, it is
enough to verify that $E_6$ is not
a fixed component.

There exists a computation sequence $\{Z_k\}_{k=0}^t$ from
$Z_0= Z+E_6$ to $Z_t=Z_K+\zmi+E_5+E_3$ such that $Z_{k+1}=Z_k+E_{i(k)}$,
$Z_kE_{i(k)}>0$, such that we add the base elements
  $E_1$, $E_2$, $E_0$, and $E_6$ in this order.
  Then
$Z_3E_{i(3)}=2$; at all the other steps $Z_kE_{i(k)}=1$.
 From the exact sequences
\[
0 \to \cO_{\tX}(-Z_{i+1})\to \cO_{\tX}(-Z_i)\to
\cO_{E_{v(i)}}(-Z_i)\to 0,
\]
we obtain that  $h^1(\cO_{\tX}(-Z_K-\zmi-E_5-E_3))+1\ge
h^1(\cO_{\tX}(-Z-E_6))$.
But, by a similar exact sequence, which connects $Z_K+\zmi$ with $Z_t$
(by adding $E_5$ and $E_3$ in this order)
$h^1(\cO_{\tX}(-Z_K-\zmi-E_5-E_3))=
h^1(\cO_{\tX}(-Z_K-\zmi))$, which is zero by Kodaira type vanishing.
Hence
\begin{equation}\label{eq:11b}
1\ge h^1(\cO_{\tX}(-Z-E_6)).
\end{equation}

Let $D=E_0+E_1+E_2+2E_6$. Then $D$ is a minimally elliptic cycle on its support and thus $h^1(D)=1$.
Since $\cO_{\tX}(-E_4^*)$ has no fixed component one has
$H^0(\cO_D(-E^*_4))\not=0$.
This and  $E_4^*D=0$ imply
that $\cO_{D}(-E_4^*)\cong \cO_D$.
On the other hand, since $2Z-3E_4^*=E_3-E_4$, we obtain that
\[
\cO_{D}(-2Z) \cong \cO_{D}(-3E_4^*)\cong \cO_D.
\]
Since $\pic(D)$ has no torsion,
we obtain $\cO_{D}(-Z)\cong \cO_D$.
Therefore,
\begin{equation}\label{eq:22b}
h^1(\cO_{\tX}(-Z))\ge h^1(\cO_{D}(-Z))=1.
\end{equation}
Finally, from \eqref{eq:11b}, \eqref{eq:22b} and  the exact sequence
\[
0\to \cO_{\tX}(-Z-E_6)\to \cO_{\tX}(-Z)\to \cO_{E_6}\to 0,
\]
we obtain  that $E_6$ cannot be a fixed component.
\end{proof}

Therefore,  the end curve condition is satisfied at all ends, and we
finished the  proof of  \thmref{t:splice}.\\

\noindent {\bf II. \ The non--Gorenstein case.} \
%We prove the following

\begin{thm}\label{t:gen}
Assume that $\zma=2\zmi$ on the minimal good resolution and $\X$ is not Gorenstein. Then $p_g\X=2$ and $\zco=E+E_5+E_6$ on the minimal good resolution.
Furthermore $\mult\X=6$ and $\emb\X=7$.
\end{thm}

Assume that $\tX$ is the minimal resolution.
Then $\zma=2E$.
By \thmref{t:chc}, we have $h^1(\cO_{2E})=p_g\X$.
Clearly $h^1(\cO_{E})=h^1(\cO_{2E})$ if and only if $p_g\X=2$; in this case, $\zco=E$ and the cohomological cycle on the minimal good resolution can be computed by \cite[2.6]{OWYcore}.

We assume that $h^1(\cO_{E})<h^1(\cO_{2E})$, namely, $p_g\X=3$;
we shall again deduce a contradiction.

From the exact sequence
\[
0\to \cO_{\tX}(-2E)\to \cO_{\tX}\to \cO_{2E}\to 0,
\]
and from $2E=\zma$, and  $\chi(2E)=-1$, we have $h^1(\cO_{\tX}(-2E))=1$.
By \eqref{eq:2E}, we have $h^1(\cO_E(-2E))=1$ too.
By duality, $h^0(\cO_E(K+3E))=1$ holds.
Hence
\begin{equation}\label{eq:ISO}\cO_E(K+3E)\cong \cO_E.\end{equation}
Note that the groups of isomorphism classes of numerically trivial line bundles on $\tX$ and $2E$ coincide, namely $H^1(\cO_{\tX})=H^1(\cO_{2E})$. Hence the triviality
of $\cO_{2E}(K+3E)$ would  contradict to the fact  that $\X$ is not Gorenstein.

We have the following exact sequence
\begin{equation}\label{eq:ext}
0\to \cO_E(K+2E) \xrightarrow{\alpha} \cO_{2E}(K+3E) \xrightarrow{\beta} \cO_E(K+3E) \to 0
\end{equation}
obtained by tensoring by $\cO_{\tX}(K+3E)$ the  exact sequence
\begin{equation}\label{eq:ext2}
0\to \cO_E(-E) \to \cO_{2E}\to \cO_E \to 0.
\end{equation}
Note that from (\ref{eq:ext2}) we obtain $h^1(\cO_E(-E))=1$ because $h^1(\cO_{2E})=3$ by the assumption.
Set $A:=\cO_E(K+2E)$, $B:=\cO_E(K+3E)$ and $N:=\cO_{2E}(K+3E)$.
Then, by (\ref{eq:ISO}), $A\cong \cO_E(-E)$ and $B\cong \cO_E$. Hence, both
exact sequences (\ref{eq:ext}) and (\ref{eq:ext2}) are extensions of
%$\cO_{\tX}$ modules
$B$ by $A$.
It is sufficient to show the following.

\begin{clm}\label{cl:M}
For any nontrivial extension
\[
 0\to A \to M  \to B \to 0
\]
of $\cO_{\tX}$-modules $B$ by $A$, we necessarily have an isomorphism  $M\cong \cO_{2E}$.
\end{clm}

Let $\Theta$ denote the bijection from the set of equivalence classes of
extensions of $B$ by $A$ to $\Ext^1(B,A)$.
This map is given by
$\Theta( 0\to A \to M  \to B \to 0)=\delta(\id_B)$, where
$\delta\: \Hom(B,B)\to \Ext^1(B,A)$ is the connecting map of the long exact sequence obtained by the functor $\Hom(B, \ \ )$.
We denote the extension \eqref{eq:ext} by $\xi$.
For any $a\in \C^*$, we  define an extension $a\cdot\xi$
by
\[
a\cdot\xi\: \quad 0\to A \xrightarrow{\alpha} N
\xrightarrow{a^{-1}\beta} B \to 0.
\]
Then $a\cdot\xi$ and $b\cdot\xi$ are quivalent if and only if $a=b$.
We show that $a\Theta(\xi)=\Theta(a\cdot\xi)$.
Here the first multiplication is in the $\C$--vector space $ \Ext^1(B,A)$.

Let us consider the injective resolution of $\xi$:
\[
\begin{array}{ccccccccc}
 &  & 0 &  & 0 &  & 0 &  &  \\
 &  & \downarrow &  & \downarrow &  & \downarrow &  &  \\
0 & \longrightarrow & A & \xrightarrow{\ \ \alpha\ \ } & N & \xrightarrow{\ \ \beta\ \ } & B & \longrightarrow & 0 \\
 &  & \downarrow &  & \downarrow &  & \downarrow &  &  \\
0 & \longrightarrow & I_0 & \xrightarrow{\ \ \alpha_0\ \ } & I_0' & \xrightarrow{\ \ \beta_0\ \ } & I_0'' & \longrightarrow & 0 \\
 &  & \downarrow &  & \downarrow &  & \downarrow &  &  \\
0 & \longrightarrow & I_1 & \xrightarrow{\ \ \alpha_1\ \ } & I_1' & \xrightarrow{\ \ \beta_1\ \ } & I_1'' & \longrightarrow & 0 \\
 &  & \downarrow &  & \downarrow &  & \downarrow &  &  \\
 &  & \vdots &  & \vdots &  & \vdots &  &
\end{array}
\]
Then the  injective resolution of $a\cdot\xi$ is obtained by replacing $\beta$ (resp. $\beta_i$) by $a^{-1}\beta$ (resp. $a^{-1}\beta_i$) in the diagram above.
We denote by $\delta_{\beta}$ the connecting map associated with $\xi$.
Applying the functor $\Hom(B, \ \ )$ to the diagram corresponding to $a\cdot \xi$, we see that $\delta_{a^{-1}\beta}(\id_B)=a\delta_{\beta}(\id_B)$.
Hence we obtain  $\Theta(a\cdot\xi)=a\Theta(\xi)$.
Since $\Ext^1(B,A)\cong H^1(\cO_E(-E))\cong \C$,
the above $\C^*$ action  on $\Ext^1(B,A)\setminus\{0\}$ is transitive, namely
$\C^* \to \C^*\Theta(\xi)$ is bijective onto $\Ext^1(B,A)\setminus\{0\}$, or
 $\C^*\Theta(\xi)=\Ext^1(B,A)\setminus\{0\}$.
Hence the extensions (\ref{eq:ext}) and (\ref{eq:ext2}) differ only by a
non--zero constant multiplication (as above) and $\cO_{2E}(K+3E)\cong \cO_{2E}$.
This implies that the singularity is Gorenstein, a contradiction.
In particular,  we have proved \clmref{cl:M} and that $p_g\X=2$.

Next we compute the multiplicity and the embedding demension.
Since $p_g\X=2$, we have $h^1(\cO_E)=h^1(\cO_{2E})=2$.
By \eqref{eq:chiE} and \eqref{eq:E2E}, we have $h^0(\cO_{2E})=1$ and $h^0(\cO_E(-E))=0$.
By \eqref{eq:1-2},
we have $h^1(\cO_{\tX}(-2E))=h^1(\cO_{\tX}(-E))=p_g\X-2=0$.
By \eqref{eq:chiE} and \eqref{eq:2E},
we have $H^0(\cO_{\tX}(-2E))\to H^0(\cO_E(-2E))$ is surjective and $h^0(\cO_E(-2E))=1$.
Therefore $\cO_{\tX}(-2E)$ has base point.
Let $g\in H^0(\cO_{\tX}(-2E))$ be a general element and $\di(g)=2E+H$.
Consider the exact sequence
\[
0\to \cO_{\tX}(-E)\xrightarrow{\times g} \cO_{\tX}(-3E) \to \cO_H(-3E) \to 0.
\]
Since $H^0(\cO_{\tX}(-3E)) \to H^0(\cO_H(-3E))$ is surjective, $\cO_{\tX}(-3E)$ has no base point.
Therefore there exists a function $h\in H^0(\cO_{\tX}(-3E))$ such that $(h)_E=3E$ and the image in $H^0(\cO_E(-3E))$ is nonzero at the base points of $\cO_{\tX}(-2E)$, namely, at $E\cap H$.
We resolve the base points and compute the multiplicity.
We have the following three cases.
Note that $HE=2$.

\begin{enumerate}
\item Assume that $H\cap E$ has two distinct points $p_1$ and $p_2$; clearly these are smooth points of $E$. Let $\phi\: Y\to \tX$ be the blowing up at $H\cap E$ and $F_i=\phi^{-1}(p_i)$.
If $Z$ denote the maximal ideal cycle on $Y$, then $Z=\phi^*(2E)+F_1+F_2$ and $\cO_Y(-Z)$ has no base points.
Therefore $\mult\X=-Z^2=6$.
Clearly the strict transform $F_0$ of $E$ is the cohomological cycle
and $\cO_{F_0}(-Z)\cong \cO_{F_0}$.
Therefore $Z$ is a $p_g$-cycle by \cite[3.10]{OWYgood}.
Hence $\emb\X=-Z^2+1=7$ by \cite[6.2]{OWYgood}.

\item Assume that $H$ intersects $E$ at a smooth point $p\in E$.
We have local coordinates $x ,y$ at $p$ such that $E=\{x=0\}$.
Then, at $p$, we may assume that $h=x^3$
and $g=x^2(y^2-xg_1)$ for some $g_1\in \C\{x,y\}$ with $g_1(0,0)\ne 0$;
therefore,
 $\m_{X,o}\cO_{\tX}=(x^3,x^2y^2)\cO_{\tX}=(x,y^2)\cO_{\tX}(-2E)$.
This base point can be resolved by two times of blowing ups;
the graph of $\di(g)$ is the following, where $F_0$ denote the strict transform of $E$.

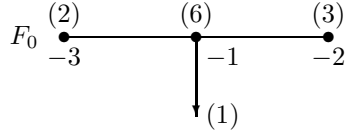
\begin{picture}(100,60)(0,20)
\put(100,60){\circle*{4}}
\put(150,60){\circle*{4}}
\put(200,60){\circle*{4}}
\put(150,68){\makebox(0,0){$(6)$}}

\put(85,60){\makebox(0,0){$F_0$}}
\put(100,68){\makebox(0,0){$(2)$}}
\put(160,30){\makebox(0,0){$(1)$}}
\put(200,68){\makebox(0,0){$(3)$}}

\put(100,52){\makebox(0,0){$-3$}}
\put(160,52){\makebox(0,0){$-1$}}
\put(200,52){\makebox(0,0){$-2$}}
\put(150,60){\vector(0,-1){30}}
\put(100,60){\line(1,0){100}}
\end{picture}
\\
By the same argument in (1), we obtain that $\mult\X=6$ and $\emb\X=7$.

\item If $H$ intersects $E$ at a singular point of $E$, then $H$ is nonsingular and the strict transform of $H$ intersects transversally one of the $(-3)$-curves  on the minimal good resolution.
We may reset our situation as follows.

Let $\tX$ be the minimal good resolution with exceptional set as in \sref{s:2} and suppose that $\zma=(g)_E=E_4^*$ and $(h)_E=3\zmi$.
By \lemref{l:E4bs}, $\cO_{\tX}(-\zmi)$ has a base point, say $P$.
Since ${\rm coeff}_{E_4}(E_4^*)=5$ and ${\rm coeff}_{E_4}(3\zmi)=6$, we see that $\m_{X,o}\cO_{\tX}=\m_P\cO_{\tX}(-\zma)$ and the base point is resolved by the blowing up at $P$.
Then $\mult\X=-\zma^2+1=6$ and $\emb\X=7$ by the same argument as in (1).
\end{enumerate}

\section{The case $\zma\ne \zmi$, $2\zmi$}\label{s:nez2z}
We assume that $\tX$ is the minimal good resolution with exceptional set as in \sref{s:2} and that $\zma\ne \zmi$, $2\zmi$ on $\tX$.
If the maximal ideal cycle on the minimal resolution is $E$,
then the base point of $\cO(-E)$ is a smooth point of $E$ and thus $\zma=\zmi$.
Hence $\coeff_{E_0}(\zma)=2$ by \proref{p:zle2e}. % where $\coeff_{E_i}$ denote the coefficient of $E_i$.
On the other hand, any anti-nef cycle on  $\tX$ with $\coeff_{E_0}=2$
is one of the following three cycles:
\[
2\zmi=2E_0^*, \quad E_1^*, \quad E_4^*.
\]
Hence we have to analyse the new cases when  $\zma$ equals either
   $\quad E_1^*$ or $ E_4^*$. Since the two cases are symmetric, in the sequel we assume that $\zma=E_4^*$.

First we start with the following lemma.

\begin{lem}\label{lem:kzmi} For  any
$\ell\geq 1$ and for analytic structure supported by $\Gamma$

(a) \  the line bundle  $\cO_{\tX}(-(\ell+2)\zmi)$ has no fixed component.

(b) \ $h^1(\cO_{\tX}(-(\ell+2)\zmi))=0$.
\end{lem}
\begin{proof}
(a)
There exists a computation sequence starting from $E^*_4+\ell\zmi+E_6$ and ending
with  $Z_K+\ell\zmi$ by adding (in this order) $E_1,\ E_2, \ E_6, E_0$, such that
at the first three steps $Z_kE_{i(k)}=1$ and at the last step $Z_kE_{i(k)}\leq 1$.
Hence $h^1(\cO_{\tX}(-E^*_4-\ell\zmi-E_6))\leq h^1(\cO_{\tX}(-Z_K-\ell\zmi))=0$.
In particular, from the exact sequence
$0\to \cO_{\tX}(-E^*_4-\ell\zmi-E_6)\to \cO_{\tX}(-E^*_4-\ell\zmi)\to \cO_{E_6}\to 0$,
$$\frac{H^0(\cO_{\tX}(-E^*_4-\ell\zmi))}{H^0(\cO_{\tX}(-E^*_4-\ell\zmi-E_6))}\cong \C.$$
Hence, there exists a function $f$ with ${\rm coeff}_{E_6}(f)=12+6\ell$,
${\rm coeff}_{E_0}(f)= 2+\ell$
and ${\rm coeff}_{E_5}(f)\geq 14+6\ell$. Symmetrically, there exists another function $f'$ with  ${\rm coeff}_{E_6}(f')\geq 14+6\ell$, ${\rm coeff}_{E_0}(f')= 2+\ell$
and ${\rm coeff}_{E_5}(f')= 12+6\ell$. Hence the divisor of $f+f'$ is $(\ell+2)\zmi$.

(b) There is a Laufer computation sequence starting from $Z_K+(\ell-1)\zmi$ and ending  with $(\ell+2)\zmi$ such that at every step $Z_kE_{i(k)}=1$.
Hence  $h^1(\cO_{\tX}(-(\ell+2)\zmi))=h^1(\cO_{\tX}(-Z_K-(\ell-1)\zmi))=0$.
\end{proof}
%We prove the following.
\begin{lem}\label{l:E4bs}
If $\zma=E_4^*$ then $p_g\X=2$ (hence $(X,o)$ is not Gorenstein), and
 $\cO_{\tX}(-E_4^*)$ has a base point.
\end{lem}
\begin{proof}
 Let $C=E_4^*-2\zmi=E_3+E_4+2E_5$.
In the exact sequence
\[
0\to \cO_{\tX}(-E_4^*)\to \cO_{\tX}(-2\zmi)\to \cO_{C}\to 0,
\]
the assumption implies $H^0(\cO_{\tX}(-E_4^*))=H^0(\cO_{\tX}(-2\zmi))$,
hence
\begin{equation}\label{eq:66}
h^1(\cO_{\tX}(-E_4^*))=1+h^1(\cO_{\tX}(-2\zmi)).
\end{equation}
Let $D=Z_K-E_4^*=E_0+E_1+E_2+2E_6$.
Then
%$C$ is the minimally elliptic cycle on its support.
%Since $-E_4^*C=0$ and $\cO_{\tX}(-E_4^*)$ has no fixed component,
we have $h^1(\cO_{D}(-E_4^*))=1$
as in the proof of \lemref{l:e3*}.
%Hence $h^1(\cO_{\tX}(-2\zmi))=0$.
From the exact sequence
\[
0\to \cO_{\tX}(-Z_K) \to \cO_{\tX}(-E_4^*)\to \cO_{D}(-E_4^*)\to 0,
\]
we obtain  $h^1(\cO_{\tX}(-E_4^*))=1$. By \eqref{eq:66}, we have $h^1(\cO_{\tX}(-2\zmi))=0$.
It follows from \eqref{eq:2-1} and \eqref{eq:pg-2} that $p_g\X=2$.

Furthermore,  $\X$ is not Gorenstein by \thmref{t:pg3}.

There exists a computation sequence $\{Z_k\}$
starting from $E^*_4+E_4$ and ending
with  $3\zmi$ such that $Z_kE_{i(k)}=2$ at two steps and otherwise $=1$.
Since $h^1(\cO_{\tX}(-3\zmi))=0$ (cf. Lemma \ref{lem:kzmi}(b)), we obtain
$
h^1(\cO_{\tX}(-E^*_4-E_4))=2
$.
In particular, from the exact sequence
\[
0\to \cO_{\tX}(-E^*_4-E_4)\to \cO_{\tX}(-E^*_4)\to \cO_{E_4}(-E^*_4)\to 0,
\]
 the image of the map $H^0(\cO_{\tX}(-E^*_4))\to H^0(\cO_{E_4}(-E^*_4))$ is 1--dimensional.
Hence $\cO_{\tX}(-E^*_4)$ has a base point.
\end{proof}

Let $f$ be the generic element of $\m_{X,o}$.
 Its divisor on $\tX$
has the form $\zma+H$, where $H$ is a cut of $E_4$ cutting it transversally in a
unique point $P$. Then in local coordinates around $P$ (with
$\{x=0\}=E$) $f$ has the form $x^5y$. By Lemma \ref{lem:kzmi}(a) there exists
a function $g$ with $(g)_E=3\zmi$, hence at $P$ with local equation $x^6$.
Therefore, $\m_{X,o}\cO_{\tX}=\m_P\cO_{\tX}(-\zma)$ and
$\mult\X=-\zma^2+1=6$.

Next,
 $\emb\X=7$ by the same argument as in (1) of the previous section.

\begin{rem}
Assume that $(X,o)$ is a singularity supported by $\Gamma$ with $p_g=2$.
Then  $h^1(\cO_{\tX}(-2\zmi))=h^1(\cO_{\tX}(-3\zmi))=0$.
Hence, from the exact sequence $0\to \cO_{\tX}(-3\zmi)\to \cO_{\tX}(-2\zmi)\to
\cO_{\zmi}(-2\zmi)\to 0$ we obtain that
$$\frac{H^0(\cO_{\tX}(-2\zmi))}{H^0(\cO_{\tX}(-3\zmi))}\cong \C.$$
Since the divisors of the analytic functions are the anti-nef cycles, and the
only  anti-nef cycles $C$ with $C\geq 2\zmi$ and $C\not\geq 3\zmi$
are $2\zmi, \ E_1^*, \ E_4^*$, out of these three cycles {\it exactly one}   appears as the divisor of an analytic function chosen by the analytic type. That divisor equals
$\zma$.
\end{rem}

%\bibliographystyle{amsplain}
%\bibliography{thebib}

\providecommand{\bysame}{\leavevmode\hbox to3em{\hrulefill}\thinspace}
\providecommand{\MR}{\relax\ifhmode\unskip\space\fi MR }
\providecommand{\MRhref}[2]{%
  \href{http://www.ams.org/mathscinet-getitem?mr=#1}{#2}
}
\providecommand{\href}[2]{#2}

\end{document}